\documentclass[12pt,twoside,reqno,psamsfonts]{amsart}

\usepackage[OT1]{fontenc}    %
\usepackage{type1cm}         %
\usepackage{amssymb}
\usepackage[dvips]{graphicx} 
\usepackage{psfrag}          
\usepackage{geometry}        
\usepackage{version}         

\numberwithin{equation}{section}

\theoremstyle{plain}
\newtheorem{theorem}{Theorem}[section]

\newtheorem{lemma}[theorem]{Lemma}

\theoremstyle{remark}
\newtheorem{remark}[theorem]{Remark}

\newtheorem*{ack}{Acknowledgement}

\theoremstyle{definition}

\newcommand{\R}{\mathbb{R}}

\newcommand{\N}{\mathbb{N}}

\newcommand{\HH}{\mathcal{H}}

\newcommand{\PP}{\mathcal{P}}
\newcommand{\QQ}{\mathcal{Q}}
\newcommand{\eps}{\varepsilon}

\newcommand{\khii}{\text{\lower -.4ex\hbox{$\chi$}}}

\newcommand{\roo}{\varrho}

\DeclareMathOperator{\dimm}{dim_M}
\DeclareMathOperator{\dimh}{dim_H}
\DeclareMathOperator{\dimp}{dim_p}
\DeclareMathOperator{\por}{por}
\DeclareMathOperator{\card}{card}
\DeclareMathOperator{\diam}{diam}

\newcommand{\yli}[2]{\genfrac{}{}{0pt}{}{#1}{#2}}

\begin{document}

\title{Nonsymmetric conical upper density and $k$-porosity}

\author{Antti K\"aenm\"aki}
\author{Ville Suomala}
\address{Department of Mathematics and Statistics \\
         P.O. Box 35 (MaD) \\
         FI-40014 University of Jyv\"askyl\"a \\
         Finland}
\email{antakae@maths.jyu.fi}
\email{visuomal@maths.jyu.fi}

\thanks{AK acknowledges the support of the Academy of Finland (project \#114821)}
\date{\today}

\begin{abstract}
  We study how the Hausdorff measure is distributed in
  nonsymmetric narrow cones in $\R^n$. As an application, we
  find an upper bound close to $n-k$ for the Hausdorff dimension of
  sets with large $k$-porosity.
  With $k$-porous sets we mean sets which have holes in $k$
  different directions on every small scale.
\end{abstract}

\maketitle

\section{Introduction}

It is a well known fact that for a set $A \subset \R^n$ with finite
$s$-dimensional Hausdorff measure, $\HH^s(A)<\infty$, we have
\begin{equation}\label{eq:bd}
  1 \le \limsup_{r \downarrow 0} \frac{\HH^s\bigl( A \cap B(x,r)
  \bigr)}{r^s} \le 2^s
\end{equation}
for $\HH^s$-almost every $x \in A$. For a proof, see, for example,
\cite[Theorem 6.2(1)]{ma}. This is analogous to the classical Lebesgue
Density Theorem. Using this fact, we know roughly how much of $A$ there
is in small balls. Mattila \cite{ma2} studied how $A$ is distributed
in such balls. He was able to estimate how much of $A$ there is near
$(n-m)$-planes. More precisely, assuming $0 \le m < s \le n$ and
denoting
\begin{equation*}
\begin{split}
  X(x,V,\alpha) &= \{ y \in \R^n : \text{dist}(y-x,V)
    < \alpha|y-x| \}, \\
  X(x,r,V,\alpha) &= X(x,V,\alpha) \cap B(x,r),
\end{split}
\end{equation*}
as $x \in \R^n$, $V \in G(n,m)$, $r>0$, and $0<\alpha\le 1$, he proved
that there exists a constant $c=c(n,m,s,\alpha)>0$ such that
\begin{equation}\label{eq:ma}
  \limsup_{r \downarrow 0} \inf_{V \in G(n,n-m)} \frac{\HH^s\bigl( A
  \cap X(x,r,V,\alpha) \bigr)}{r^s} \geq c
\end{equation}
for $\HH^s$-almost every $x \in A$ whenever $A \subset \R^n$ is such
that $\HH^s(A) < \infty$. Here $G(n,m)$ denotes the collection of all
$m$-dimensional linear subspaces of $\R^n$, see \cite[\S
  3.9]{ma}. Actually \eqref{eq:ma} is
just a special case of Mattila's result, as his theorem can be applied
also for more general cones, see \cite[Theorem 3.3]{ma2}.


In Theorem
\ref{thm:nonsymmetric_conical} we show that if $A$ is as above, then
it cannot be concentrated in too small regions, not even inside the
cones $X(x,r,V,\alpha)$. More precisely, denoting
\begin{equation*}
\begin{split}
  H(x,\theta) &= \{ y \in \R^n : (y-x)\cdot\theta > 0 \}, \\
  H(x,\theta,\eta) &= \{ y \in \R^n : (y-x)\cdot\theta
    > \eta|y-x| \},
\end{split}
\end{equation*}
for $x \in \R^n$, $\theta \in S^{n-1}$, and $0<\eta\le 1$, we prove
under the same assumptions as in \eqref{eq:ma} that
there exists a constant $c=c(n,m,s,\alpha,\eta)>0$ such that
\begin{equation*}
  \limsup_{r \downarrow 0} \inf_{\yli{\theta \in S^{n-1}}{V \in
  G(n,n-m)}} \frac{\HH^s\bigl( A \cap X(x,r,V,\alpha) \setminus
  H(x,\theta,\eta) \bigr)}{r^s} \geq c
\end{equation*}
for $\HH^s$-almost every $x \in A$. Here $S^{n-1}$ denotes the unit sphere of $\R^n$. To help the geometric visualization, it might be helpful to take $\alpha$ and $\eta$ close to $0$ and $\theta \in V \cap S^{n-1}$. Our method gives also a more elementary
proof for \eqref{eq:ma} and it can also be used to obtain similar
results for more general measures, see Theorem \ref{thm:h}.

The nonsymmetric conical upper density theorem is essential in our
application to $k$-porous sets, that is, the sets with $\por_k > 0$, see
\eqref{eq:poronen_joukko}. The notation of porosity, or
$1$-porosity using our terminology, has arisen from the study of
dimensional estimates related, for example, to the boundary behavior of
quasiconformal mappings. See Koskela and Rohde \cite{KR}, Martio and
Vuorinen \cite{MV}, Sarvas \cite{sar}, Trocenko \cite{T}, and
V\"ais\"al\"a \cite{V}. The dimensional properties of $1$-porous sets are well known. Using a version of \eqref{eq:ma}, Mattila showed that if porosity is close to its maximum value $\tfrac12$, then the dimension cannot be much bigger than $n-1$. More precisely,
\begin{equation}\label{eq:1por}
\sup\{ s>0 :
\por_1(A) > \roo \text{ and } \dimh(A)>s \text{ for some } A \subset
\R^n \} \longrightarrow n-1
\end{equation}
as $\roo\to\tfrac{1}{2}$. Here $\dimh$ refers to the Hausdorff
dimension. Later Salli \cite{sa2} generalized this result for the
Minkowski dimension, and found the correct asymptotics. The concept of
$1$-porosity has also been generalized for measures, and it leads to
similar kind of dimension bounds. See J\"arvenp\"a\"a and
J\"arvenp\"a\"a \cite{JJ} and references therein.

Motivated by the fact that each $V \in G(n,n-1)$ has maximal
$1$-porosity, we introduce a porosity condition which describes also
sets whose dimension is smaller than $n-1$.
For any integer $0 < k \le n$, $x \in \R^n$, $A \subset \R^n$, and $r > 0$ we set
\begin{align} \label{eq:porodef1}
  \por_k(A,x,r) = \sup\{ \roo : \; &\text{there are }\notag
  z_1,\ldots,z_k \in \R^n \text{ such that } \notag \\
  &B(z_i,\roo r)\subset B(x,r)\setminus
  A \text{ for every } i, \\
  &\text{and }(z_i-x)
  \cdot(z_j-x)=0 \text{ for } i \neq j \}. \notag
\end{align}
Here $\cdot$ is the inner product.
The \emph{$k$-porosity of $A$ at a point $x$} is defined to be
\begin{equation*}
  \por_k(A,x) = \liminf_{r \downarrow 0} \por_k(A,x,r),
\end{equation*}
  and the \emph{$k$-porosity of $A$} is given by
\begin{equation} \label{eq:poronen_joukko}
  \por_k(A) = \inf_{x \in A} \por_k(A,x).
\end{equation}

This means that $k$-porous sets have holes in $k$ orthogonal directions
near each of its points in every small scale.
We shall now give a concrete example where $k$-porosity occurs
naturally. Suppose $0<\lambda<\tfrac{1}{2}$ and  let
$C_\lambda\subset\R$ be the
usual $\lambda$-Cantor set, see \cite[\S 4.10]{ma}. It is clearly a
$1$-porous set with $\por_1(C_\lambda) \approx \tfrac{1}{2} -
\lambda$. Mattila's result \eqref{eq:1por}
implies that
$\dimh(C_\lambda) \to 0$ as $\por_1(C_\lambda) \to \tfrac{1}{2}$. Of
course, we could obtain the same information just by calculating the
Hausdorff dimension of the self-similar set $C_\lambda$ and letting
$\lambda \to 0$, but our aim was to provide the reader with an
illustrative example. The sets $C_\lambda \times C_\lambda \subset \R^2$
and $C_\lambda \times C_\lambda \times [0,1] \subset \R^3$ are clearly
$2$-porous with $\por_2 \approx \tfrac{1}{2} - \lambda$. For these
sets \eqref{eq:1por} does not give any reasonable dimension
bound. However, it would
be desirable to see, also in terms of porosity, that $\dimh(C_\lambda
\times C_\lambda) \to 0$ and $\dimh(C_\lambda \times C_\lambda \times
[0,1]) \to 1$ as $\lambda \to 0$. This follows as an immediate
application of Theorem \ref{thm:iso_poro}. Using our nonsymmetric
conical upper density theorem, we show that
\begin{equation*}
\sup\{ s>0 :
\por_k(A) > \roo \text{ and } \dimh(A)>s \text{ for some } A \subset
\R^n \} \longrightarrow n-k
\end{equation*}
as $\roo \to \tfrac{1}{2}$.
Observe also that in the proof of Theorem \ref{thm:iso_poro} the
orthogonality in \eqref{eq:porodef1} plays no r\^ole
and we may replace it by an assumption of a
uniform lower bound for the angles between $z_i-x$ and the $(k-1)$-plane
spanned by vectors $z_j-x$, $i\neq j$.

Let us now discuss the situation when porosity is small. It is well
known (for example, see \cite{MV}) that if $A\subset\R^n$ with
$\por_1(A,x,r) \ge \roo > 0$ for all $x\in A$ and $0<r<r_0$, then
\begin{equation}\label{eq:bigporo}
\dimm(A)<n-c\roo^n,
\end{equation}
where $c>0$ depends only on $n$, and $\dimm$ refers to the Minkowski
dimension, see \cite[\S 5.3]{ma}.
It might be possible to get a better estimate if $\por_1$ is replaced
by $\por_k$ for some $k>1$, but this condition does not feel very
natural if the size of the holes is small.
However, if $V \in G(n,m)$ is fixed and
the condition $\por_1(A,x,r) \ge \roo$ is replaced by
\begin{equation*}
    \sup\bigl\{\roo' : B(z,\roo' r)\subset B(x,r)\setminus A\text{ for some
    }z\in V+\{x\}\bigr\}\geq\roo,
\end{equation*}
then $n$ in \eqref{eq:bigporo} can be replaced by $m$,
see Theorem \ref{thm:sporo}. This is a rather immediate consequence of
\eqref{eq:bigporo}, but our main point is to give a simple
proof for \eqref{eq:bigporo} using iterated function systems.

\begin{ack} 
  The authors are indebted to Professor Pertti Mattila for his
  valuable comments for the manuscript. The authors thank also
  Esa J\"arvenp\"a\"a, Maarit J\"arvenp\"a\"a, Pekka Koskela, Tomi
  Nieminen, Kai Rajala and Eero Saksman for useful
  discussions during the preparation of this article.
\end{ack}

\section{Nonsymmetric conical upper density}

We shall
first prove a density theorem for nonsymmetric regions and then
prove our main theorem by using a
similar argument on $(n-m)$-planes. The proofs rely on the
following geometric fact.

\begin{figure}
  \psfrag{z}{$z$}
  \psfrag{w}{$w$}
  \begin{center}
    \includegraphics[width=0.5\textwidth]{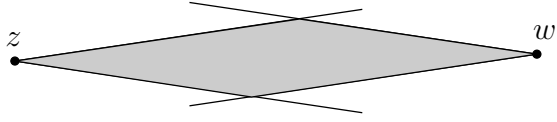}
  \end{center}
  \caption{All points lying on the gray region form a large angle
  with points $z$ and $w$.}
  \label{fig:kulma}
\end{figure}

\begin{lemma}\label{lemma:EF}
  For given $0<\beta<\pi$, there is $q=q(n,\beta) \in \N$ such that in
  any set of $q$ points in $\R^n$, there are always three points which
  determine an angle between $\beta$ and $\pi$.
\end{lemma}

\begin{remark}
  Erd\H{o}s and F\"uredi \cite{EF} have shown that for the smallest
  possible choice of $q$ it holds that
  \begin{equation*}
    2^{(\pi/(\pi-\beta))^{n-1}}\leq
    q(n,\beta)\leq2^{(4\pi/(\pi-\beta))^{n-1}} + 1.
  \end{equation*}
  For the convenience of the reader we shall give below a different
  proof which establishes the existence of some such $q$. The estimate
  that we get here for $q$ is, however, quite bad compared to the best
  possible one.
\end{remark}

\begin{proof}
  Let $A$ be a set of points in $\R^n$ so that all angles formed by
  its points are less than $\beta$. Let us fix $0<\eta<1$ and cover
  $\R^n\setminus\{0\}$ by cones
  $C_i=H(0,\theta_i,\eta)$, $i \in \{ 1,2,\ldots,k \}$, where the constant $k=k(n,\eta) \in \N$ depends only on $n$ and $\eta$. To visualize the situation, note that if $\beta$ is close to $\pi$, then $\eta$ is close to $1$ and cones $C_i$ are very narrow. To simplify the notation, we
  denote $C_{i,y}=C_i+\{y\}$ for $y \in \R^n$.

  For any index $i_1i_2\cdots i_j$, where $j\in\N$ and
  $i_m\in\{1,2,\ldots,k\}$ for $1\leq m\leq j$, we define sets
  $A_{i_1i_2\cdots i_j}$ in the following way: We begin by fixing $x\in
  A$ and setting $A_i=A\cap C_{i,x}$ for $1\leq i\leq k$. If
  $A_{i_1i_2\cdots i_j}$ has been defined, we choose $y \in
  A_{i_1i_2\cdots i_j}$ and define $A_{i_1i_2\cdots i_j
  l}=A_{i_1i_2\cdots i_j}\cap C_{l,y}$ for $1\leq l\leq k$
  (if $A_{i_1i_2\cdots i_j}$ is empty, then so is $A_{i_1i_2\cdots i_j
  l}$). We refer to $y$ as the corner of $A_{i_1i_2\cdots i_j l}$. It
  follows directly from the definition of the sets $A_{i_1i_2\cdots
  i_j}$ that
  \begin{equation*}
    \card A_{i_1i_2\cdots i_j}\leq 1+\sum\limits_{l=1}^{k}\card
    A_{i_1i_2\cdots i_jl}.
  \end{equation*}
  Iterating this, we get
  \begin{equation}\label{eq:rekursio}
    \card A\leq\sum\limits_{j=0}^{k}k^j + \sum_{i_1i_2\cdots i_k}
    \sum\limits_{l=1}^{k} \card A_{i_1i_2\cdots i_k l}.
  \end{equation}

  The main point of the proof is the observation that if
  $\eta=\eta(\beta)$ is chosen to be close enough to $1$ in the beginning,
  then the following is true: If $z$ and $w$ are the corners of
  $A_{i_1i_2\cdots i_j}$ and $A_{i_1i_2\cdots i_ji_{j+1}\cdots i_m}$,
  respectively, and if $z\in C_{i_m, w}$, then
  $A\cap C_{i_j,z}\cap C_{i_m,w}=\emptyset$. See Figure
  \ref{fig:kulma}.
  It follows by induction from the above fact that for given
  $A_{i_1i_2\cdots i_j}$ we have
  \begin{equation*}
    \card\{l : A_{i_1i_2\cdots i_jl}\neq\emptyset\}\leq k-j.
  \end{equation*}
  In particular, $A_{i_1i_2\cdots i_{k+1}}=\emptyset$ for any choice of
  $i_1i_2\ldots i_{k+1}$. Combined with \eqref{eq:rekursio}, this
  gives $\card A \leq \sum_{j=0}^{k} k^j$. This number depends
  only on $k=k(n,\beta)$ and the claim follows.
\end{proof}

For $0 < \eta \le 1$ we define
\begin{align*}
  t(\eta) &= \sqrt{\frac{\eta^2 + 4}{\eta^2}}, \\
  \gamma(\eta) &= \frac{1}{t(\eta)}.
\end{align*}
Notice that $t(\eta) \ge 2$ and $\eta/\sqrt{5} \le \gamma(\eta) \le
\eta/2$.

\begin{lemma} \label{thm:etamato}
  Suppose that $y \in \R^n$, $\theta \in S^{n-1}$, $0<\eta\le 1$,
  $t=t(\eta)$, and $\gamma=\gamma(\eta)$. If $z \in \R^n \setminus
  \bigl( B(y,tr) \cup H(y,\theta,\gamma) \bigr)$, then
  \begin{equation*}
    B(z,r) \cap H(y,\theta,\eta) = \emptyset.
  \end{equation*}
\end{lemma}

\begin{proof}
  Take $w \in \R^n$ such that it maximizes $(w-y)\cdot\theta/|w-y|$ in
  the closure of $B(z,r)$. It suffices to prove that
  $(w-y)\cdot\theta/|w-y| < \eta$, see Figure \ref{fig:skaalalemma}.
  It is straightforward to check that 
  $\eta\sqrt{s^2-1}\geq 1+\gamma s$ when $s\geq t$.
  Denoting now $s = |y-z|/r$, we have $s\geq t>1$ 
  and thus
 \begin{equation*}
 \begin{split}
    (w-y)\cdot\theta &< r+\gamma|y-z| = (1+\gamma s)r \\
    &\le \eta\sqrt{s^2-1}r = \eta|w-y|,
  \end{split}
  \end{equation*}
  which finishes the proof.
\end{proof}

\begin{figure}
  \psfrag{y}{$y$}
  \psfrag{(s2-1)r}{$\sqrt{s^2-1}r$}
  \psfrag{sr}{$sr$}
  \psfrag{w}{$w$}
  \psfrag{r}{$r$}
  \psfrag{z}{$z$}
  \psfrag{theta}{$\theta$}
  \begin{center}
    \includegraphics[width=0.65\textwidth]{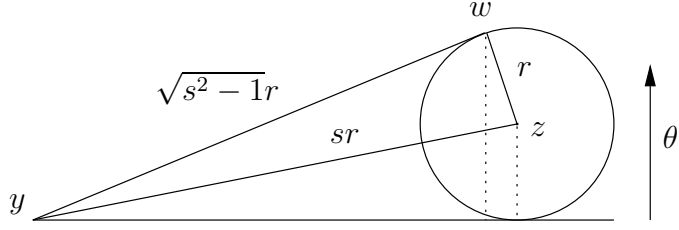}
  \end{center}
  \caption{Illustration for the proof of Lemma \ref{thm:etamato}.}
  \label{fig:skaalalemma}
\end{figure}

\begin{theorem} \label{thm:tylppakulma}
  Suppose $0 < \eta \le 1$ and $0 < s \le n$. Then there is a constant
  $c=c(n,s,\eta)>0$ such that
  \begin{equation*}
    \limsup_{r \downarrow 0} \inf_{\theta\in S^{n-1}} \frac{\HH^s\bigl( A
    \cap B(x,r) \setminus H(x,\theta,\eta) \bigr)}{r^s} \geq c
  \end{equation*}
  for $\HH^s$ almost every $x \in A$ whenever $A\subset\R^n$ with
  $\HH^s(A)<\infty$.
\end{theorem}

\begin{proof}
  Take $c>0$ and assume that there exists a Borel set $B \subset \R^n$
  with $\HH^s(B) > 0$ such that for each $x \in B$ and $0<r<r_0$ there
  is $\theta \in S^{n-1}$ for which
  \begin{equation} \label{eq:ristiriita}
    \HH^s\bigl( B \cap B(x,r) \setminus H(x,\theta,\eta) \bigr) < cr^s.
  \end{equation}
  It suffices to find a positive lower bound for $c$ in terms of $n$,
  $s$, and $\eta$.

  Using \eqref{eq:bd}, and replacing $B$ by a suitable subset if
  necessary, we may assume that
  \begin{equation} \label{eq:yksrasti}
    \HH^s\bigl( B \cap B(x,r) \bigr) < 2^{s+1}r^s
  \end{equation}
  for all $0<r<r_0$ and $x \in B$. Moreover, using the lower estimate
  of \eqref{eq:bd}, we find
  $0<r<r_0/3$ and $x \in B$ such that
  \begin{equation} \label{eq:kaksrasti}
    \HH^s\bigl( B \cap B(x,r) \bigr) > \tfrac{1}{2}r^s.
  \end{equation}

  Set $t = t(\eta)$, $\gamma=\gamma(\eta)$, and take $0 < \delta < 1$.
  Let us fix $\beta<\pi$ such that the opening angle of
  $H(x,\theta,\gamma)$ is smaller than $\beta$, and let $q=q(n,\beta)$
  be as in Lemma \ref{lemma:EF}.
  We may cover the set $B \cap B(x,r)$ by
  $4^n\delta^{-n}$ balls of radius $\delta r$ with centers
  in $B$. Using  \eqref{eq:kaksrasti}, we notice that there exists
  $x_1 \in B \cap B(x,r)$ such that
  \begin{equation*}
    \HH^s\bigl( B \cap B(x_1,\delta r) \bigr) > 4^{-n}\delta^n
    2^{-1}r^s.
  \end{equation*}
  The set $B \cap B(x,r) \setminus B(x_1,t\delta r)$ can also be covered
  by $4^n\delta^{-n}$ balls of radius $\delta r$ with centers in
  $B$. Whence, using \eqref{eq:yksrasti} and \eqref{eq:kaksrasti},
  \begin{equation*}
    \HH^s\bigl( B \cap B(x,r) \setminus B(x_1,t\delta r) \bigr) > (\tfrac{1}{2}
    - 2^{s+1}t^s\delta^s)r^s.
  \end{equation*}
  If $\tfrac{1}{2} - 2^{s+1}t^s\delta^s > 0$, we find $x_2 \in B \cap
  B(x,r) \setminus B(x_1,t\delta r)$ for which
  \begin{equation*}
    \HH^s\bigl( B \cap B(x_2,\delta r) \bigr) >
    4^{-n}\delta^n(\tfrac{1}{2} - 2^{s+1}t^s\delta^s)r^s.
  \end{equation*}
  Choosing $\delta = \delta(n,s,\eta) > 0$ small enough and continuing
  in this manner, we find $q$ points $x_1,\ldots,x_q \in B \cap B(x,r)$
  with $|x_i - x_j| \geq t\delta r$ for $i \ne j$, such that for each
  $i \in \{ 1,\ldots,q \}$ we have
  \begin{equation} \label{eq:kolmerasti}
  \begin{split}
    \HH^s\bigl( B \cap B(x_i,\delta r) \bigr) &> 4^{-n}\delta^n\bigl(
    \tfrac{1}{2} - (q-1)2^{s+1}t^s\delta^s \bigr)r^s \\ &
    =: c(n,s,\eta)(3r)^s,
  \end{split}
  \end{equation}
  where $c(n,s,\eta)>0$.

  According to Lemma \ref{lemma:EF}, we may choose three points
  $y,y_1,y_2$ from the set $\{ x_1,\ldots,x_q \}$ such that for each
  $\theta \in S^{n-1}$ there is $i\in\{1,2\}$ for which $y_i \in \R^n \setminus
  \bigl( B(y,t\delta r) \cup H(y,\theta,\gamma) \bigr)$. We obtain, using
  Lemma \ref{thm:etamato}, that for each $\theta \in S^{n-1}$ there is
  $i\in\{1,2\}$ such that
  \begin{equation*}
    B(y_i,\delta r) \subset B\bigl( y,2(1+\delta)r \bigr) \setminus
    H(y,\theta,\eta).
  \end{equation*}
  Thus, applying (\ref{eq:kolmerasti}), we have
  \begin{equation*}
    \HH^s\bigl( B \cap B(y,3r) \setminus H(y,\theta,\eta) \bigr) >
    c(n,s,\eta)(3r)^s
  \end{equation*}
  for all $\theta \in S^{n-1}$. Recalling \eqref{eq:ristiriita},
  we conclude that $c \geq c(n,s,\eta)$. The proof is finished.
\end{proof}

\begin{theorem} \label{thm:nonsymmetric_conical}
  Suppose $0 < \alpha,\eta \le 1$ and $0 \le m < s \le n$. Then there
  is a constant $c=c(n,m,s,\alpha,\eta)>0$ such that
  \begin{equation*}
    \limsup_{r \downarrow 0} \inf_{\yli{\theta \in S^{n-1}}{V \in
    G(n,n-m)}} \frac{\HH^s\bigl( A \cap X(x,r,V,\alpha) \setminus
    H(x,\theta,\eta) \bigr)}{r^s} \geq c
  \end{equation*}
  for $\HH^s$ almost every $x \in A$ whenever $A\subset\R^n$ with
  $\HH^s(A)<\infty$.
\end{theorem}

\begin{proof}
  For any $V,W\in G(n,n-m)$, we set $d(V,W) = \sup_{x \in V \cap S^{n-1}}
  \text{dist}(x,W)$. With this metric $G(n,n-m)$ is a compact metric
  space, see Salli
  \cite{Sa}. Defining for each $V \in G(n,n-m)$ a set $\{ W : d(V,W) <
  \alpha/2 \}$ we notice
  that a finite
  number of these sets is still a cover. We assume that the sets assigned to
  the planes $V_1,\ldots,V_l$, where $l=l(n,m,\alpha)$, cover
  $G(n,n-m)$. For any $W$, it holds that $d(V_i,W) < \alpha/2$ with some
  $i \in \{1,\ldots,l\}$.
  This implies $X(0,V_i,\alpha/2) \subset X(0,W,\alpha)$. Thus, for each $W \in
  G(n,n-m)$, there is $i$ such that
  \begin{equation} \label{eq:inkluusio}
    X(x,r,W,\alpha) \supset X(x,r,V_i,\alpha/2)
  \end{equation}
  for all $r>0$ and $x \in \R^n$. We shall prove that if $A\subset\R^n$
  with $\HH^s(A)<\infty$, then
  \begin{equation*}
    \limsup_{r \downarrow 0} \inf_{\yli{\theta \in
    S^{n-1}}{i \in \{ 1,\ldots,l\} }} \frac{\HH^s\bigl( A \cap
    X(x,r,V_i,\alpha/2) \setminus H(x,\theta,\eta) \bigr)}{r^s} \geq
    c(n,m,s,\alpha,\eta)
  \end{equation*}
  for $\HH^s$ almost every $x \in A$
  from which the claim follows easily by using (\ref{eq:inkluusio}).

  Take $c>0$ and assume that there is a Borel set $B \subset \R^n$
  with $\HH^s(B)>0$ such that for each $x \in B$ and $0<r<r_0$ there
  are $i$ and $\theta \in S^{n-1}$ for which
  \begin{equation*}
    \HH^s\bigl( B \cap X(x,r,V_i,\alpha/2) \setminus H(x,\theta,\eta)
    \bigr) < cr^s.
  \end{equation*}
  According to \eqref{eq:bd} we may assume that
  \begin{equation}\label{eq:x}
    \HH^s\bigl( B \cap B(x,r) \bigr) < 2^{s+1}r^s
  \end{equation}
  for all $0<r<r_0$ and $x \in B$. Using the lower estimate of
  \eqref{eq:bd}, we find $0 <
  r < r_0/3$ and $x \in B$ such that
  \begin{equation} \label{eq:alaarvio}
    \HH^s\bigl( B \cap B(x,r) \bigr) > \tfrac{1}{2}r^s.
  \end{equation}
  Next we define
  \begin{equation} \label{eq:antiteesi}
  \begin{split}
    B_i = \bigl\{ z \in B : \HH^s\bigl( B \cap X(z,3r,V_i,\alpha/2)
    \setminus &H(z,\theta,\eta) \bigr) < c(3r)^s \\ &\text{for some } \theta\in
    S^{n-1} \bigr\}.
  \end{split}
  \end{equation}
  Since $\bigcup_{i=1}^l B_i = B$, we infer from \eqref{eq:alaarvio}
  that there is $i_0 \in \{ 1,\ldots,l \}$ for which
  \begin{equation*}
    \HH^s\bigl( B_{i_0} \cap B(x,r) \bigr) > 2^{-1}l^{-1}r^s.
  \end{equation*}

   Let $t=\max\{5/\alpha, t(\eta)\}$, choose $q=(n,\eta)$ as in
   the proof of
  Theorem \ref{thm:tylppakulma}, and define $0<\eps<1$ so that
  \begin{equation}\label{eq:xxx}
    4^{-m}2^{-1}l^{-1}\eps^m-(q-1)2^{s+1}t^s\eps^s=4^{-m-1}l^{-1}\eps^m\,,
  \end{equation}
  recall that $s>m$ so that this is possible.
  Since the set
  $(V_{i_0}^\bot + \{ x \}) \cap B(x,r)$ may
  be covered by $4^m\eps^{-m}$ balls of radius $\eps r$, there
  exists $y \in (V_{i_0}^\bot + \{ x \}) \cap B(x,r)$ such that
  \begin{equation}\label{eq:xx}
    \HH^s\bigl( B_{i_0} \cap B(x,r) \cap P_{V_{i_0}^\bot}^{-1}(B(y,\eps r))
    \bigr) > 4^{-m}2^{-1}l^{-1}\eps^m r^s.
  \end{equation}
  We now argue as in the proof of Theorem \ref{thm:tylppakulma}
  above. We first observe that the slice
  $S=B_{i_0} \cap B(x,r) \cap P_{V_{i_0}^\bot}^{-1}\bigl( B(y,\eps r) \bigr)$ may be
  covered by $c_{1}^{-1}\eps^{m-n}$ balls of radius $\eps r$ for a
  constant $c_{1}=c_{1}(n,m)>0$. Then we use \eqref{eq:xx},
  \eqref{eq:x}, and \eqref{eq:xxx}  to find points
  $\{x_1,\ldots,x_q\}\in S$ such that
  $|x_i-x_j|\geq t\varepsilon r$ whenever $i\neq j$ and
  \begin{equation} \label{eq:xxxx}
  \begin{split}
    \HH^s\bigl( S \cap B(x_i,\eps r) \bigr) &>  c_1\eps^{n-m}\left(4^{-m}2^{-1}l^{-1}\eps^m
    r^s-(q-1)2^{s+1}t^s\eps^sr^s\right) \\ &= c_2 (3r)^s
  \end{split}
  \end{equation}
  for all $i$. Here
  $c_2=c_2(n,m,s,\alpha,\eta)=c_13^{-s}4^{-m-1}l^{-1}\varepsilon^m$.
  Now the same geometric argument as in the proof of Theorem
  \ref{thm:tylppakulma} implies that there is a point $z
  \in\{x_1,\ldots,x_q\}$ such that for each
  $\theta\in S^{n-1}$ we may find
  $w\in\{x_0,\ldots,x_q\}\setminus\{z\}$ so that
  \begin{equation*}
    B(w,\eps r) \subset
    B\bigl(z,(2+\eps)r\bigr) \setminus \bigl(H(z,\theta,\eta)\cap
    B(z,4\eps r/\alpha)\bigr).
  \end{equation*}
  Since also
  \begin{equation*}
    P_{V_{i_0}^\perp}^{-1}\bigl(B(y,\eps r)\bigr) \cap
    B(z,3r)\setminus B(z,4\eps
    r/\alpha) \subset X(z,3r,V_{i_0},\alpha/2),
  \end{equation*}
  see Figure \ref{fig:nonsymmetric_conical}, we get
  \begin{equation*}
  \inf\limits_{\theta\in S^{n-1}} \HH^s\bigl(B\cap
  X(z,3r,V_{i_0},\alpha/2)\setminus H(z,\theta,\eta)\bigr) \geq
  c_2(3r)^s.
  \end{equation*}
  by \eqref{eq:xxxx}. Now $z \in B_{i_0}$ and we conclude, using
  \eqref{eq:antiteesi},
  that $c \geq c_2=c_2(n,m,s,\alpha,\eta)$. This completes the proof.
\end{proof}

\begin{figure}
  \psfrag{w}{$w$}
  \psfrag{y}{$y$}
  \psfrag{z}{$z$}
  \psfrag{4er/a}{$4\eps r/\alpha$}
  \psfrag{x}{$x$}
  \psfrag{r}{$r$}
  \psfrag{V+x}{$V+\{ x \}$}
  \begin{center}
    \includegraphics[width=0.45\textwidth]{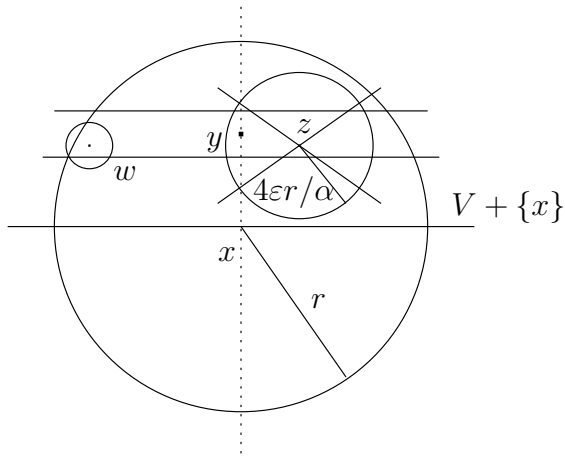}
  \end{center}
  \caption{Illustration for the proof of Theorem \ref{thm:nonsymmetric_conical}.}
  \label{fig:nonsymmetric_conical}
\end{figure}

\begin{remark}
  Inspecting the proofs, one can read explicit expressions for the constants in Theorems \ref{thm:tylppakulma} and \ref{thm:nonsymmetric_conical}. In Theorem \ref{thm:tylppakulma}, one gets $c \ge 2^{c_1/(-s\eta^{n-1})}$ and in Theorem \ref{thm:nonsymmetric_conical}, one obtains $c \ge \alpha^{c_3/(s-m)} \, 2^{c_2/((m-s)\eta^{n-1})}$. The constants $0<c_1,c_2,c_3<\infty$ here depend only on $n$. The estimates obtained in this way are probably rather far from being optimal, although the best values are not known.
\end{remark}

Our method can be applied also in a more general setting. A similar
proof as above gives the following result. If $\mu$ is a measure on
$\R^n$, $h\colon(0,r_0)\rightarrow(0,\infty)$, and $x\in\R^n$, we
define $\overline{D}(\mu,x)$ and $\underline{D}(\mu,x)$ as the lower
and upper limits, respectively, of the ratio $\mu\bigl( B(x,r) \bigr)/h(r)$ as $r\downarrow 0$.

\begin{theorem}\label{thm:h}
Suppose $0 \le m < n$ and $h\colon (0,r_0)\rightarrow (0,\infty)$
is a function with
\begin{equation}\label{eq:h}
  \frac{h(\eps r)}{\eps^m h(r)} \longrightarrow 0 \qquad
    \text{uniformly for all } 0<r<r_0
\end{equation}
as $\eps \downarrow 0$. Let $\mu$ be a measure on $\R^n$ with
$\overline{D}(\mu,x)<\infty$ for $\mu$-almost all $x\in\R^n$.
For every  $0<\alpha,\eta\leq 1$, there is a constant
$c=c(n,m,h,\alpha,\eta)>0$ such that
\begin{equation*}
  \limsup_{r \downarrow 0} \inf_{\yli{\theta \in S^{n-1}}{V \in
  G(n,n-m)}} \frac{\mu\bigl(X(x,r,V,\alpha) \setminus
  H(x,\theta,\eta) \bigr)}{h(r)} \ge c\overline{D}(\mu,x)
\end{equation*}
for $\mu$-almost every $x\in\R^n$.
\end{theorem}

Let us make few comments related to the above theorem. Suppose that
$h$ fulfills  condition  \eqref{eq:h}. Let $\HH_h$ be the generalized
Hausdorff measure which is constructed using $h$ as a gauge function,
see \cite[\S 4.9]{ma}. If $\mu=\HH_h|_A$, where $\HH_h(A)<\infty$, then
$\overline{D}(\mu,x)<\infty$ for $\mu$-almost every $x\in\R^n$, and
thus Theorem \ref{thm:h} can be applied.

There are many natural gauge functions, such as $h(r)=r^s\log(1/r)$
where $m<s<n$, which satisfy \eqref{eq:h}. However, some interesting
cases, such as $h(r)=r^m/\log(1/r)$, are not covered by this
condition.

It seems to be unknown whether a similar result as Theorem \ref{thm:h}
holds if one replaces the condition $\overline{D}(\mu,x)<\infty$ by
$\underline{D}(\mu,x)<\infty$. The most interesting example falling
into this category is obtained when $\mu=\mathcal{P}^s|_A$ and
$h(r)=r^s$, where $\mathcal{P}^s(A)<\infty$ and $m<s<n$. Here $\PP^s$ denotes the $s$-dimensional packing measure, see \cite[\S 5.10]{ma}. See also
Suomala \cite{S} for related theorems.

\section{Sets with large $k$-porosity}

Mattila \cite{ma2} proved Theorem \ref{thm:nonsymmetric_conical} in
the case $m=n-1$. Using this, he obtained the desired dimension bounds
for $1$-porous sets, see \eqref{eq:1por}. Our result for $k$-porous
sets follows applying a similar argument.

For $\sqrt{2}-1 < \roo < \tfrac{1}{2}$ we define
\begin{align*}
  t(\roo) &= \frac{1}{\sqrt{1-2\roo}}\,,\\
  \delta(\roo) &=\frac{1-\roo-\sqrt{\roo^2+2\roo-1}}{\sqrt{1-2\roo}}\,.
\end{align*}
Notice that $\delta(\roo) \to 0$ as $\roo \to \tfrac{1}{2}$.

\begin{lemma} \label{thm:delta_lemma}
  Suppose $x \in \R^n$, $r > 0$, $\sqrt{2}-1 < \roo < \tfrac{1}{2}$,
  $t=t(\roo)$, and $\delta=\delta(\roo)$. If $z \in \R^n \setminus \{ x \}$ is such that $B(z,\roo tr) \subset B(x,tr)$, then
  \begin{equation*}
    H(x+\delta r\theta,\theta) \cap B(x,r) \subset B(z,\roo tr),
  \end{equation*}
  where $\theta = (z-x)/|z-x|$.
\end{lemma}

\begin{proof}
  To simplify the notation, we assume $r=1$, $x=0$, and
  $\theta=e_1=(1,0,\ldots,0)$. This will not affect the
  generality. Let $y\in B(0,1)\setminus B(z,\roo t)$. We have to show
  that
  \begin{equation}\label{eq:notin}
    y\notin H(x+\delta\theta,\theta).
  \end{equation}
  By the Pythagorean Theorem we have
  \begin{equation*}
    |z-y_1|=\sqrt{|z-y|^2-|y-y_1|^2} \geq \sqrt{(\roo t)^2-1}.
  \end{equation*}
  Using this, we obtain
  \begin{equation*}
    y_1=|z|-|z-y_1|\leq t-\roo t-\sqrt{(\roo t)^2-1} = \delta,
  \end{equation*}
  which implies \eqref{eq:notin}.
\end{proof}

\begin{theorem} \label{thm:iso_poro}
  Suppose $0<k\le n$. Then
  \begin{equation*}
    \sup\{ s>0 : \por_k(A) > \roo \text{ and } \dimh(A)>s \text{ for
    some } A \subset \R^n \} \longrightarrow n-k
  \end{equation*}
  as $\roo \to \tfrac{1}{2}$.
\end{theorem}

\begin{proof}
  Assume on the contrary that there exists $s > n-k$ such that for
  each $\sqrt{2}-1 < \roo < \tfrac{1}{2}$ there is a set $A_\roo$ for
  which $\dimh(A_\roo) > s$ and $\por_k(A_\roo) > \roo$.
  Take $\sqrt{2}-1 < \roo < \tfrac{1}{2}$ and such a set $A_\roo$. Now
  $A_\roo$ has a subset $B$ for which $\dimh(B) > s$ and
  $\por_k(B,x,r) > \roo$ for all $x \in B$ and $0<r<r_0$ with some
  $r_0>0$. Clearly also the closure of $B$ satisfies these
  conditions. Thus there is a closed set $F \subset \overline{B}$ (for
  example, use \cite[Theorem 5.4]{fa}) such that $0<\HH^s(F)<\infty$ and
  \begin{equation*}
    \por_k(F,x,r) > \roo \qquad \text{for all } x \in F \text{ and }
    0<r<r_0.
  \end{equation*}
  Therefore, for any $x \in F$ and $0<r<r_0/t$, there are
  $z_1,\ldots,z_k\in\R^n$ such that $B(z_i,\roo tr)\subset B(x,tr)\setminus F$
  for $i=1,\ldots,k$, and $(z_i-x) \cdot (z_j-x) = 0$ for $i \neq j$.
  Put $\theta_i=(z_i-x)/|z_i-x|$.
  \begin{figure}
    \psfrag{t1}{$\theta_1$}
    \psfrag{t2}{$\theta_2$}
    \psfrag{dr}{$\delta r$}
    \psfrag{x}{$x$}
    \psfrag{t}{$\theta$}
    \begin{center}
      \includegraphics[width=0.45\textwidth]{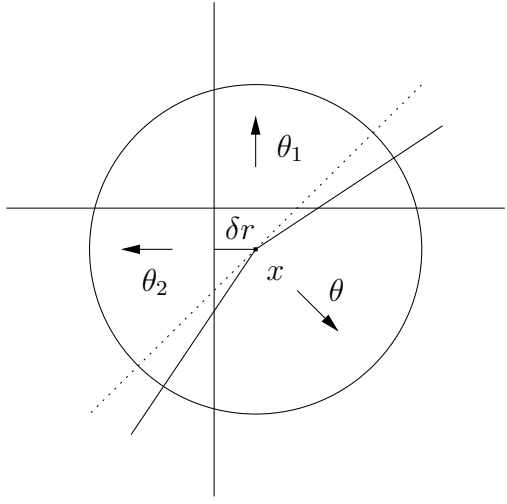}
    \end{center}
    \caption{Illustration for the proof of Theorem \ref{thm:iso_poro}: the situation when $n=2$ and $k=2$.}
    \label{fig:iso_poro}
  \end{figure}
  Applying now Lemma \ref{thm:delta_lemma}
  we have $H(x+\delta r\theta_i,\theta_i) \cap B(x,r) \subset
  B(z_i,\roo tr)$ for every $i$. Here $t=t(\roo)$ and
  $\delta=\delta(\roo)$. Thus
  \begin{equation} \label{eq:poistot}
    F \cap B(x,r) \subset \bigcap_{i=1}^k B(x,r) \setminus H(x+\delta
    r\theta_i,\theta_i).
  \end{equation}
  Put $\theta = -\tfrac{1}{\sqrt{k}} \sum_{i=1}^k \theta_i$ and take
  $V \in G(n,k)$ such that $\theta_i \in V$ for every $i$. Now
  choosing $\alpha$ and $\eta$ small enough, we have, using
  \eqref{eq:poistot}, that
  \begin{equation} \label{eq:poistot2}
    F \cap X(x,r,V,\alpha) \setminus H(x,\theta,\eta) \subset
    B(x,2 n^{1/2}\delta r).
  \end{equation}
  Observe that the choice of $\alpha$ and $\eta$ does not depend on
  $\delta$ and hence not on $\roo$ either.
  Figure \ref{fig:iso_poro} illustrates the situation. Using Theorem
  \ref{thm:nonsymmetric_conical}, we may fix
  $x \in F$ and $0<r<r_0/t$ for which
  \begin{equation} \label{eq:arvio1}
    \HH^s\bigl( F \cap X(x,r,V,\alpha) \setminus H(x,\theta,\eta) \bigr)
    \ge c2^{2s+1}n^{s/2} r^s,
  \end{equation}
  where $c=c(n,k,s,\alpha,\eta)>0$. By \eqref{eq:bd} we may assume that also
  \begin{equation}\label{eq:ub}
    \HH^s\bigl( F \cap B(x,2n^{1/2}\delta r) \bigr) \le
    2^{2s+1} n^{s/2} \delta^s r^s.
  \end{equation}
  Combining \eqref{eq:poistot2}--\eqref{eq:ub}, we have
  $c 2^{2s+1} n^{s/2} r^s \le 2^{2s+1} n^{s/2} \delta^s r^s$ and hence
  \begin{equation*}
    s \le \frac{\log{c}}{\log\delta(\roo)}.
  \end{equation*}
  But the constant $c$ does not depend on $\roo$, and thus
  $\log c/\log \delta(\roo)\rightarrow 0$ as $\roo\rightarrow\tfrac{1}{2}$
  giving a contradiction.
\end{proof}

\section{Sets with small porosity}

Finally, let us briefly discuss the situation when porosity is small.
The proof of the following theorem can be found for example in  Martio
and Vuorinen \cite{MV}. We
shall give here a different proof, and then show how the theorem can
be improved when more information on the location of the holes is
given.

\begin{theorem}\label{thm:MV}
  Let $A\subset\R^n$ be bounded and suppose that $\por_1(A,x,r) \ge \roo$
  for all $x\in A$ and $0<r<r_0$. Then $\dimm(A) < n-c\roo^n$, where
  $c>0$ depends only on $n$.
\end{theorem}

\begin{proof}
  We may assume that $r_0=1$ and $A\subset[0,1]^n$. Let us denote by
  $\mathcal{Q}_j$ the collection of all closed dyadic cubes
  $Q\subset[0,1]^n$ with side length $2^{-j}$. Let $l$ be the smallest
  integer with $2^{-l+2}<\roo/\sqrt{n}$. It is easy to see that for
  any $Q\in\mathcal{Q}_j$ there is $Q'\in\mathcal{Q}_{j+l}$ such that
  $Q' \subset Q$ and $Q'\cap A=\emptyset$. Let us fix one such $Q'$
  for each $Q\in\bigcup_{j=1}^{\infty}\mathcal{Q}_j$. Next we define
  a set $B\subset[0,1]^n$ by setting
  \begin{equation} \label{eq:beendef}
    B=[0,1]^n\setminus \bigcup\limits_{j=0}^{\infty}
    \bigcup\limits_{Q\in\mathcal{Q}_j}Q'.
  \end{equation}
  For any $Q\in\mathcal{Q}_j$, let $x_Q$ be the corner of $Q$ which is
  nearest to the origin, and let $\widetilde{Q} = \{x_Q\} +
  [0,2^{-j-1}]^n$. If we define $E \subset [0,1]^n$ by setting
  \begin{equation*}
    E = [0,1]^n \setminus \bigcup\limits_{j=0}^{\infty}
    \bigcup\limits_{Q\in\mathcal{Q}_j} \textrm{int}\widetilde{Q},
  \end{equation*}
  \begin{figure}
    \psfrag{1}{$f_1$}
    \psfrag{2}{$f_2$}
    \psfrag{3}{$f_3$}
    \psfrag{4}{$f_4$}
    \psfrag{5}{$f_5$}
    \psfrag{6}{$f_6$}
    \psfrag{7}{$f_7$}
    \psfrag{8}{$f_8$}
    \psfrag{9}{$f_9$}
    \begin{center}
      \includegraphics[width=0.4\textwidth]{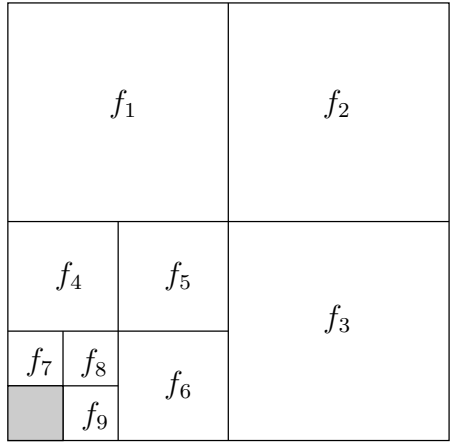}
    \end{center}
    \caption{Similitudes $f_k$ in the proof Theorem \ref{thm:MV} when $n=2$ and $l=3$.}
    \label{fig:sim}
  \end{figure}%
  where $\textrm{int}$ denotes the interior of a given set,
  then obviously $\dimm(E) \geq \dimm(B)$, see also \cite{K}.
  The set $E$ is the limit set of the iterated
  function system defined by the similitudes $f_k$,
  $k \in \{ 1,2,\ldots,l(2^n-1) \}$, see Figure \ref{fig:sim}.
  For any $i \in \{ 1,\ldots,l \}$, there are $2^n-1$ similitudes among
  $\{f_k\}_{k=1}^{l(2^n-1)}$ with contraction ratio $2^{-i}$. Since
  the open set condition is clearly satisfied, the dimension
  $s=\dimm(E) = \dimh(E)$ is given by
  \begin{equation} \label{eq:self_sim_dim}
    (2^n-1)\sum\limits_{i=1}^{l}2^{-is}=1,
  \end{equation}
  see Hutchinson \cite[\S 5]{H}. This reduces to
  \begin{equation*}
    2^{n-s}=1+(2^n-1)2^{-(l+1)s}
  \end{equation*}
  and since $\log_2(1+x) \ge x/\bigl((1+x)\log 2\bigr)$ for $x \ge 0$,
  we have
  \begin{equation*} \begin{split}
    s&=n-\log_2\bigl(1+(2^n-1)2^{-(l+1)s}\bigr) \\
    &\leq n-\log_2\bigl(1+(1-2^{-n})2^{-ln}\bigr) \\
    &\leq n-\frac{2}{5\log 2} 2^{-ln}\leq n-c\roo^n,
    \end{split} \end{equation*}
  where $c=\bigl(2/(5\log 2)\bigr)2^{-3n}n^{-n/2}$. Because $A\subset B$ and
  $\dimm(B)\leq\dimm(E)=s$, we conclude that also $\dimm(A)\leq
  n-c\roo^n$.
\end{proof}

In the above proof, the use of the self-similar set $E$ is not a necessity, but it concretizes the situation. The key point in the proof is that for any cube $Q \subset \R^n$ which is small enough, one can find subcubes $Q_1,\ldots,Q_{l(2^n-1)} \subset Q$ such that $A \cap Q \subset \bigcup_{i=1}^{l(2^n-1)} Q_i$ and $\sum_{i=1}^{l(2^n-1)} \diam(Q_i)^s = \diam(Q)^s$, where $s$ is given by \eqref{eq:self_sim_dim}. From this the desired dimension bound follows easily.

\begin{remark}
  In a sense the above result is the best possible one. There is a
  constant $c'=c'(n)>0$ and sets $A_\roo$, $0<\roo<1/2$, with
  $\dimh(A_\roo) > n-c'\roo^n$, and $\por_1(A_\roo,x,r) \ge \roo$ for
  all $r>0$ and
  $x\in\R^n$. See, for example, Koskela and Rohde \cite{KR}, or
  estimate the Hausdorff dimension of the set $E$ from below.
\end{remark}

\begin{theorem}\label{thm:sporo}
  Let $A\subset\R^n$ be bounded and suppose that there is $V\in
  G(n,m)$ such that for all $x\in A$ and $0<r<r_0$ one has
  \begin{equation}\label{eq:sporo}
    \sup\bigl\{\roo' : B(z,\roo' r)\subset B(x,r)\setminus A\text{ for some
    }z\in V+\{x\}\bigr\}\geq\roo.
  \end{equation}
  Then $\dimm(A)<n-c\roo^m$, where $c>0$ depends only on $n$ and $m$.
\end{theorem}

\begin{proof}
  Without losing the generality we may assume that
  $V=\R^m= \{ x \in \R^n : x_{m+1}=x_{m+2}=\ldots=x_n=0\}$,
  $r_0=\sqrt{n}$, and $A\subset[0,1]^n$. Let $\mathcal{Q}_j$ be, as
  before, the collection of all closed dyadic cubes $Q\subset[0,1]^n$
  with side length $2^{-j}$, and let $\mathcal{\widetilde{Q}}_j = \{
  P_V(Q) : Q \in \QQ_j \}$ and $\mathcal{Q}'_j = \{ P_{V^\bot}(Q) : Q
  \in \QQ_j \}$. Here $P_V$ is the orthogonal projection onto
  $V$. Furthermore, let $l$ be the smallest integer with
  $2^{-l+2}<\roo/\sqrt{n}$.

  We define a set $E = E_{l,m} \subset V$ as in the proof of Theorem
  \ref{thm:MV}. For $j\in\N$ we let $a_j=a_{j,l,m}$ denote the minimum
  number of cubes from the collection $\mathcal{\widetilde{Q}}_j$ that
  are needed to cover $E$. The proof of Theorem \ref{thm:MV}
  yields that
  \begin{equation}\label{eq:aj}
    \lim\limits_{j\rightarrow\infty}
    \frac{\log a_j}{\log(2^j)}\leq m-c2^{-ml},
  \end{equation}
  where $c > \tfrac{1}{2}$ is an absolute constant.

  It is straightforward to convince oneself of the following fact: If
  $\widetilde{Q}\in\mathcal{\widetilde{Q}}_j$ and
  $Q' \in \mathcal{Q}'_{j+l}$,
  then there is $Q\in\mathcal{Q}_{j+l}$ such that
  $P_{V^\perp}(Q) = Q'$,
  $P_V (Q)\subset
  \widetilde{Q}$, and $A\cap Q=\emptyset$. From this observation it
  follows that given $Q'\in\mathcal{Q}'_j$, only $a_j$ cubes from the
  collection $\{Q\in\mathcal{Q}_j : P_{V^\perp}(Q)=Q'\}$ touch the set
  $A$. Thus only $2^{j(n-m)}a_j$ cubes from the collection
  $\mathcal{Q}_j$ are needed to cover $A$. Using \eqref{eq:aj}, we
  calculate
  \begin{equation*}
  \begin{split}
    \dimm(A) &\leq \limsup\limits_{j\downarrow 0}
    \frac{\log(2^{j(n-m)}a_j)}{\log(2^j)} =
    n-m+\limsup\limits_{j\downarrow 0}
    \frac{\log a_j}{\log(2^j)} \\
    &\leq n-c2^{-ml}\leq n-c2^{-3m}n^{-m/2}\roo^m.
  \end{split}
  \end{equation*}
  The proof is finished.
\end{proof}

\begin{remark}
  Suppose that $V\in G(n,m)$ is fixed and $A\subset\R^n$ is such that
  \eqref{eq:sporo} holds for every $x\in A$ and $0<r<r_x$, where $r_x>0$
  depends on the point $x$. It follows immediately from Theorem
  \ref{thm:sporo} that $\dimh(A)\leq\dimp(A)\leq n-c\roo^m$, where $c$
  is as in Theorem \ref{thm:sporo} and $\dimp$ denotes the packing
  dimension, see \cite[\S 5.9]{ma}. The above dimension estimates are
  also sharp. Consider, for example, sets of the form
  $E\times\R^{n-m}$, where $E\subset\R^m$ is as in the proof of
  Theorem \ref{thm:MV}.
\end{remark}

\begin{remark}
  After the submission of this article in May 2004, there has been considerable progress in the study of conical densities and porosities. Most notably, the question posed after Theorem \ref{thm:h} has been answered positively in \cite{KaenmakiSuomala2008}. For improvements of Theorems \ref{thm:iso_poro} and \ref{thm:MV}, see \cite{JarvenpaaJarvenpaaKaenmakiSuomala2005} and \cite{JarvenpaaJarvenpaaKaenmakiRajalaRogovinSuomala2007}, respectively.
\end{remark}


\end{document}